\documentclass[anon]{opt2023}
\usepackage{ulem}
\usepackage{tcolorbox}

\title[AGD: A guaranteed bound for a heuristic restart strategy.]{Accelerated gradient descent: A guaranteed bound\\ for a heuristic restart strategy}

\author{\qquad \qquad \qquad
	Walaa M.\ Moursi\thanks{
		Department of Combinatorics and Optimization, 
		University of Waterloo,
		Waterloo, Ontario N2L~3G1, Canada.	
		E-mail: \texttt{walaa.moursi@uwaterloo.ca}}
	\qquad
	Viktor Pavlovic\thanks{
		Department of Combinatorics and Optimization, 
		University of Waterloo,
		Waterloo, Ontario N2L~3G1, Canada.
		E-mail: \texttt{v2pavlov@uwaterloo.ca}}
	\qquad
	Stephen A. Vavasis\thanks{
		Department of Combinatorics and Optimization, 
		University of Waterloo,
		Waterloo, Ontario N2L~3G1, Canada. E-mail:
		\texttt{vavasis@uwaterloo.ca}.}
}

\newcommand{\knn}{\ensuremath{{k\in{\mathbb N}}}}

\DeclareMathOperator*{\Argmin}{Argmin}
\newcommand{\scal}[2]{\left\langle{#1},{#2}  \right\rangle}
\providecommand{\siff}{\Leftrightarrow}
\newcommand{\Id}{\ensuremath{\operatorname{Id}}}
\providecommand{\abs}[1]{\lvert#1\rvert}
\RestyleAlgo{ruled}
\crefname{equation}{}{equations}
\crefname{item}{}{items}

\begin{document}
	
	\maketitle
	
	\begin{abstract}%
		The $\mathcal{O}(1/k^2)$ convergence rate in function value of accelerated gradient descent is optimal, but there are many modifications that have been used to speed up convergence in practice. Among these modifications are restarts, that is, starting the algorithm with the current iteration being considered as the initial point. We focus on the adaptive restart techniques introduced by O'Donoghue and Cand\`es, specifically their gradient restart strategy. While the gradient restart strategy is a heuristic in general, 
		we prove that applying gradient restarts 
		preserves and in fact improves the $\mathcal{O}(1/k^2)$ bound,
		hence establishing function value convergence, for one-dimensional functions.
		Applications of our results to separable and nearly separable functions are presented.
	\end{abstract}

	
	
	\section{Introduction}
	In 1983, Nesterov introduced acceleration to the classical gradient descent method  \cite{Nesterov1983AMF}. This acceleration method involves the addition of an extrapolation based on previous iterates, the new iterate is then computed as a classical gradient step from this extrapolation. Accelerated gradient descent (AGD) is also known as gradient descent with momentum. This is because the extrapolation step is computed using an increasing sequence of scalars. Due to this added momentum AGD is a non-monotonic method as extrapolated points can overshoot the minimizer and cause an increase in function value.\\
	\\
	Several restarting algorithms have been proposed recently as discussed in \cref{sec:sec21}. Restarting means that the momentum term is set to its initial value of zero, and the current iteration is used as the new starting point. In effect, this deletes the memory used previously to compute the extrapolation steps.\\
	\\
	O'Donoghue and Cand\`es introduced an adaptive restart strategy that does not depend on any properties of the objective function \cite{O'Donoghue2015Jun}. The authors offer two heuristic schemes to restart AGD. The first is a function value scheme that restarts whenever non-monotonicity is detected, and the second restarts when the gradient makes an obtuse angle with the direction of the iterates. The latter scheme uses only already computed information. The authors suggest that the gradient based scheme is favorable as it is more numerically stable. While both schemes perform well in practice and tend to drastically improve the performance of AGD, unlike AGD, there is no proof of function value convergence for the restarted scheme. The authors have provided an analysis when the objective is quadratic and suggest that this analysis carries over to a quadratic region around the minimizer. There have however been examples, such as in the work by Hinder and Lubin \cite{hinder2020generic} who give a simple function for which the restarts drastically underperform until the iterates get close to the minimizer \cite{hinder2020generic}.
	Recall that 
	a function $f:\R^n \rightarrow \R$
	is \textit{L-smooth} if it has Lipschitz continuous gradient, i.e.,  for all $x$ and 
	$y $ in $ \R^n$ we have,
	\begin{equation}
		\norm{\gf{x} -\gf{y}} \leq L\norm{x-y}.
	\end{equation}
	Throughout the remainder
	of this paper, we assume that
	\begin{equation}
		\label{e:f:assmp}
		\text{$f:\R^n \rightarrow \R$ 
			is convex and $L$-smooth, where $L>0$,
			and $S:=\Argmin f\neq \varnothing$ }
	\end{equation}
	and we set 
	\begin{equation}
		\label{e:Smu:assmp}
		\text{$f^*:=\min f(\R^n)$.}
	\end{equation}
	In this paper we
	provide analysis for the case $n=1$. Our results, primarily \cref{theorem:1Re} and \cref{prop:mRe} below, reveal that applying the gradient restart improves
	the classical right-hand side bound of AGD and hence preserves the 
	function value convergence
	of the iterates of the restarted scheme. Unlike the O'Donoghue and Cand\`es restart scheme, our analysis varies
	in that if the restart condition
	(see \cref{eqn:1dGrad} below) is satisfied then we restart at 
	$x_{k+1}$ rather than $x_k$. Computational experiments (see Appendix~\ref{sec:D} and Appendix~\ref{sec:E}) suggest that in practice restarting from $x_{k+1}$ offers slightly better practical performance for $n\ge 1$.

	\section{Framework and Related Works}\label{sec:Sec2}
	Recalling \cref{e:f:assmp},
	our framework for AGD is that of an unconstrained convex optimization problem
	
	\begin{equation}
		\min_{x \in \R^n} f(x).
	\end{equation}
	Let $(t_k)_{\knn}$ be a sequence of scalars such that $t_0 =1$
	and $(\forall k\ge 1)$ (see, e.g., \cite{FISTA})
	\begin{equation}\label{eqn:FISTAseqA}
		t_k \geq \tfrac{k+2}{2} \geq 1 =t_0
		\;\text{and}\;  t_k^2 \geq t^2_{k+1} - t_{k+1}.
	\end{equation}
	Two popular choices for the sequence of $t_k$ are given by,
	\begin{equation}
		t_{k+1} = \tfrac{1 + \sqrt{1 + 4t_k^2}}{2}, \; \; \; t_k = \tfrac{k+2}{2}.
		\label{eq:tchoice}
	\end{equation}
	Let $x_0 \in \R^n$,
	and set $y_0 = x_0$.
	Throughout this paper we set
	\begin{equation}
		T = {\Id - \tfrac{1}{L}\nabla f}.
	\end{equation}
	
	Classical AGD updates $x_0$ 
	$(\forall k\in \mathbb{N})$ as follows:
	\begin{subequations}
		\label{e:classicAGD}
		\begin{align}
			x_{k+1}&
			:= T(y_k)
			\\
			y_{k+1} &:= x_{k+1} + \tfrac{t_k -1}{t_{k+1}}\left(x_{k+1}-x_k\right).
		\end{align}
	\end{subequations}
	We set 
	\begin{equation}
		x^*:=P_{S}(x_0)\in S,
	\end{equation}
	where $P_S$ denotes the orthogonal projection 
	(this is also known as the closest point mapping)
	onto the set $S$ which is convex, closed, and by assumption (see \cref{e:f:assmp}) nonempty.
	
	Note that the classical upper bound for AGD is $(\forall \knn)$
	\begin{equation}
		f(x_k) - f^* \leq \frac{2L\norm{x_0-x^*}^2}{(k+1)^2}.
	\end{equation}
	We now recall the following useful fact which will be used later.
	\begin{fact}{\rm\bf (see \cite[Theorem~10.16]{BeckFirstOrder})}
		\label{fact:1}
		Let $x \in \R^n$ and let  $y\in \R^n$. Then
		\begin{equation}
			f(x) - f(T(y)) \geq \tfrac{L}{2}\norm{x - T(y)}^2 - \tfrac{L}{2}\norm{x-y}^2.
		\end{equation}
	\end{fact}

	\subsection{Restarts}\label{sec:sec21}
	We will focus on adaptive restart strategies. These strategies restart the algorithm according to whether some condition is satisfied at the current iteration. 
	In \cite{O'Donoghue2015Jun}
	O'Donoghue and Cand\`es  proposed 
	two  conditions  for restarts:  the \textit{function value} and \textit{gradient based} conditions. 
	The function value condition imposes 
	restarts whenever 
	\begin{equation}
		f(x_{k+1}) > f(x_k).
	\end{equation}
	On the other hand, the gradient based condition imposes 
	restarts whenever 
	\begin{equation}
		\label{e:gradrescond}
		\scal{\gf{y_k}}{x_{k+1}-x_k} > 0.
	\end{equation}
	This indicates that there is a disagreement in the direction the iterates should go in, and thus it would be good to restart. The function scheme requires evaluating the objective at $x_k$ which may be expensive, while the gradient based scheme requires no new computations whatsoever. \\
	\\
	Giselsson and Boyd introduced different adaptive restart schemes \cite{giselsson2014monotonicity}. Their framework considers the  composite model where the objective function
	is of the form $F = f+g$ where   $g$ is convex, lower semicontinuous, and proper. This is the setup for FISTA \cite{FISTA}. Note that AGD can be seen as a special case of FISTA when we set $g \equiv 0.$ They suggest that non-monotonicity of function values is a good indicator of when to restart, and provide a new condition which implies non-monotonicity, namely 
	\begin{equation}
		\scal{y_{k-1} - x_k}{x_{k+1}-\tfrac{1}{2}(x_k + y_{k-1})}>0.
	\end{equation}
	In addition to this new restart test, they provide a convergence rate to a modified version of the restarted AGD algorithm. A notable difference 
	between their restart scheme
	and the ones introduced in \cite{O'Donoghue2015Jun}
	is that they do not 
	reset the parameter sequence of $t_k$'s. 
	This allowed  the authors to deduce an $O(1/k^2)$ convergence rate 
	similar to classical AGD.\\
	\\
	Another common assumption is that of quadratic growth \cite{Necoara2019May}, \cite{GradBasedFISTARe},  \cite{QuadGrowthFeroq} and \cite{aujol2023parameterfree}.  
	We say that 
	$f$ satisfies 
	a
	quadratic growth condition if
	there exists $\mu>0$ such that 
	\begin{equation}
		f(x) - f^* \geq \frac{\mu}{2}\norm{x-P_S(x)}^2.
	\end{equation}
	Necoara, Nesterov, and Glineur provide a fixed restart strategy for such functions \cite{Necoara2019May}. They derive the optimal restart frequency based on the knowledge of the parameter $\mu.$\\
	\\
	The growth parameter $\mu$ is rarely known exactly, and thus many restart strategies for this class of functions work by estimating it. In Feroq and Qu \cite{QuadGrowthFeroq}, the authors show that their adaptive restart strategy which requires an estimation of the growth parameter, satisfies a linear convergence rate. In Aujol et al. \cite{aujol2023parameterfree} the authors use restarting and adaptive backtracking to estimate both $L$ and $\mu$ and prove a linear convergence rate without any prior knowledge of the actual values of the parameters, thus they are able to run FISTA without any parameters.\\
	\\
	Hinder and Lubin provide another adaptive restart strategy based on a potential function that is always computable  \cite[Theorem~3]{hinder2020generic}. Their analysis requires the assumption  that $f$ is strongly convex, and this assumption cannot be reduced to quadratic growth. According to their computational experiments,
	their restart method is comparable to the heuristic methods in  \cite{O'Donoghue2015Jun} in general.
	However, they provided an example 
	of a strongly convex and $L-$smooth separable function on $\R^n$
	where the gradient based restarts \cref{e:gradrescond} 
	performed worse than classical AGD
	while their proposed restart scheme
	performed better.
	In passing, we point out that
	our analysis identifies the drawback of applying restart condition \cref{e:gradrescond} in this case.
	As a consequence of our analysis, we show that for separable functions the condition
	\cref{e:gradrescond} should be applied in a separable way by checking \cref{e:gradrescond} for each coordinate.

	\section{Contributions}
	Recalling \cref{e:f:assmp}, we 
	provide analysis for the case $n=1$.
	We focus only on the gradient restart condition \cref{e:gradrescond}.
	In the one-dimensional case
	this condition boils down to
	\begin{equation}\label{eqn:1dGrad}
		f^\prime(y_k) (x_{k+1}-x_k) >0.
	\end{equation}
	Observe that if for some 
	$k \in  \{0,1,2\dots\}$ we had 
	$f^\prime(x_k) =0$, then $x_k$ is a minimizer and we stop the algorithm. 
	\begin{remark}
		\label{rem:ourres}
		In passing we point out that 
		in our scheme,
		we employ the restart condition \cref{eqn:1dGrad} 
		with a slight modification to that proposed in \cite{O'Donoghue2015Jun}.
		Indeed, if \cref{eqn:1dGrad}  is satisfied
		our algorithm keeps $x_{k+1}$
		the same, it resets $y_{k+1} := x_{k+1}$ and $t_{k+1} := 1.$ Note that this is different from the scheme proposed in \cite{O'Donoghue2015Jun} which sets $x_{k+1} := x_k, y_{k+1} := x_k, t_{k+1} := 1$ whenever \cref{eqn:1dGrad} is satisfied. We have observed that our modification offers slightly better performance with the restarts (see \cref{fig:AGD_Quad}
		and \cref{fig:AGD_Huber} below).
	\end{remark}
	Our proposed restart algorithm in the one-dimensional case is given in  Algorithm~1  below.
	
	\begin{tcolorbox}
		\textbf{Algorithm~1: AGD with Restarts}
		\\
		\textbf{Input:} $x_0\in \R$, 
		$(\widetilde{t}_k)_\knn$ is a sequence of scalars that satisfies
		\cref{eqn:FISTAseqA}.
		\\
		\textbf{Initialization:} $y_0=x_0$, $({t}_k)_\knn:=(\widetilde{t}_k)_\knn$.
		\\
		\textbf{General step:}
		For $k\ge 0$ 
		update $x_k $ via 
		\begin{align*}
			x_{k+1}&
			:= T(y_k)
		\end{align*}
		If
		\begin{equation*}
			f^\prime(y_k) (x_{k+1}-x_k) \le 0
		\end{equation*}
		update $y_k $ via 
		\begin{align*}
			y_{k+1} &:= x_{k+1} + \tfrac{t_k -1}{t_{k+1}}\left(x_{k+1}-x_k\right).
		\end{align*}
		Else 
		restart as follows:
		$y_{k+1}:=x_{k+1}$, $(t_{k+m+1})_{m\in\mathbb{N}}:=(\widetilde{t}_m)_{m\in\mathbb{N}}$.
		\label{alg:AgdRE}
	\end{tcolorbox}
	
	\begin{proposition}
		\label{prop:GySign}
		Let $x_0\in \R$
		and let $(x_k)_\knn$ 
		and 
		$(y_k)_\knn$ be the sequences obtained
		by Algorithm~1. 
		Let $\knn$ 
		be such that 
		for all $\overline{k}\le k$
		the restart condition 
		\cref{eqn:1dGrad} has not been satisfied.
		Then the following hold.
		\begin{enumerate}[(i)]
			\item 
			\label{prop:GySign:i}
			If $x_k < x^*$ then 
			the restart condition in \cref{eqn:1dGrad} is satisfied at iteration $k$ if and only if $x_k <x^* < x_{k+1}.$
			\item 
			\label{prop:GySign:ii}
			If $x_k > x^*$ then 
			the restart condition in \cref{eqn:1dGrad} is satisfied at iteration $k$ if and only if $x_k >x^* > x_{k+1}.$
		\end{enumerate}
	\end{proposition}
	\begin{proof}
		See Appendix~\ref{sec:1}.
	\end{proof}
	\begin{remark}
		\label{rem:totalord}
		A direct consequence of Proposition~\ref{prop:GySign} is that
		we get an ordering of the iterates $x_k$ between restarts.
		Indeed, suppose we have two restarts, $r_1$ and $r_2$ with $r_2 > r_1$.
		If $y_{r_1} < x^*$ then for all 
		$k \in \{r_1,\dots r_2 -1\}$ we have $f^\prime(y_k) < 0$.
		Since we do not restart at any of these iterations we must have that $x_{k+1} - x_k > 0$. Similarly,  if $y_{r_1} > x^*$ we deduce that $x_{k+1} - x_k < 0$.
	\end{remark}
	
	We will also state the following proposition:
	\begin{proposition}\label{prop:xnew_bound}
		Let $x_0\in \R$
		and let $(x_k)_\knn$ 
		and 
		$(y_k)_\knn$ be the sequences obtained
		by Algorithm~1. 
		Let $\knn$ 
		be such that 
		for all $\overline{k}\le k-1$
		the restart condition 
		\cref{eqn:1dGrad} has not been satisfied
		and that
		$k$ is the first iteration where the restart condition is satisfied, i.e., 
		
		\begin{equation}
			\label{e:condprop}
			f^\prime(y_k)(x_{k+1} -x_k) >0.   
		\end{equation}
		Define $z_{k+1}:= (1-t_k)x_{k} + t_k x_{k+1}$.
		Then,
		$$|z_{k+1} - x^*| \geq t_k | x_{k+1} - x^*|.$$
	\end{proposition}
	\begin{proof}
		Suppose that $x_0 < x^*$. By Proposition~\ref{prop:GySign} we have $x_k < x^* < x_{k+1}.$ 
		It follows from this, the definition of $z_{k+1}$
		and \cref{eqn:FISTAseqA} 
		that 
		\begin{equation}
			\begin{split}
				|z_{k+1} - x^*| &= |(1-t_k) x_k + t_k x_{k+1} - x^*|\\
				&=|(1-t_k)(x_k-x^*) + t_k(x_{k+1}-x^*)|\\
				&= (t_k-1)|x_k-x^*| + t_k|x_{k+1}-x^*|\\
				&\geq t_k|x_{k+1} - x^*|.
			\end{split}
		\end{equation}
		The case where $x_0 > x^*$ is shown similarly.
	\end{proof}
	\begin{remark}
		A careful look at the proof of 
		Proposition~\ref{prop:xnew_bound}
		in view of Remark~\ref{rem:totalord}
		reveals that the conclusion
		of Proposition~\ref{prop:xnew_bound}
		generalizes to any number of restarts.
	\end{remark}
	We now recall the following useful fact.
	\begin{fact}{\rm\bf (see, e.g., \cite[Theorem 30.4]{Walaa})}
		\label{lemma:decProp}
		Let $x_0\in \R$
		and let   $(x_k)_\knn$,
		and $(t_k)_\knn$ be given by AGD \cref{e:classicAGD}. 
		For 
		$k \in  \{0,1,2\dots\}$
		we set 
		\begin{subequations}
			\begin{align}
				z_{k+1}&:= (1-t_k)x_{k} + t_k x_{k+1},
				\\
				\delta_k &:= f(x_k) - f(x^*).
			\end{align}
			\label{eq:def:zdelta}
		\end{subequations}
		Then we have the following monotonicity property
		\begin{equation}
			t^2_k\delta_{k+1} + \tfrac{L}{2}\abs{z_{k+1}-x^*}^2 \leq t^2_{k-1}\delta_k + \tfrac{L}{2}\abs{z_k-x^*}^2.
			\label{eq:mono}
		\end{equation}
	\end{fact}
	We are now ready for our main result. For restarts using the gradient condition \cref{eqn:1dGrad}, we have the following theorem.
	\begin{theorem}[a single restart.]\label{theorem:1Re}
		Let $x_0\in \R$
		and let $(x_k)_{k \in \mathbb{N}},(y_k)_{k \in \mathbb{N}},(t_k)_{k \in \mathbb{N}}$ be given by Algorithm~1. 
		Suppose that iteration $r$ is the first iteration where 
		we have 
		\begin{equation}
			f^\prime (y_r)(x_{r+1}-x_r) > 0,
		\end{equation}
		and that iteration $\overline{r}$ is the second iteration where 
		we have $f^\prime (y_{\overline{r}})(x_{\overline{r}+1}-x_{\overline{r}}) > 0$.
		Then  
		the following hold.
		\begin{enumerate}[{\rm(i)}]
			\item 
			$(\forall k \leq r+1)$
			we have $f(x_k) -f^* \leq \tfrac{2L(x_0-x^*)^2}{(k+1)^2}$.
			\item 
			$f(x_{r+2})\le f(x_{r+1})$.
			\item 
			$(\forall k\in \{r+3,\ldots, \overline{r}+1\})$ 
			we have 
			\begin{equation}
				f(x_k) -f^* \leq \left( \frac{4}{(k-r)(r+2)}\right)^2\frac{L(x_0-x^*)^2}{2} \leq \frac{2L(x_0-x^*)^2}{(k+1)^2}.
				\label{eq:newbound}
			\end{equation}
		\end{enumerate}
	\end{theorem}
	\begin{proof}
		See Appendix~\ref{sec:2}.
	\end{proof}

	\begin{theorem}[multiple restarts.]\label{prop:mRe} 
		Let $x_0\in \R$ and set $(\forall \knn )$ $t_k=\tfrac{k+1}{2}$.
		Let $(x_k)_{k \in \mathbb{N}}$ be given by  Algorithm~1. 
		Let $p\in \mathbb{N}$. Suppose that $r_p$ is the 
		$p^{\text{\rm th}}$-iteration such that 
		\begin{equation}
			\label{e:condprop:pth}
			f^\prime(y_{r_p})(x_{r_p+1} -x_{r_p}) >0   .
		\end{equation}
		Then $(\forall p\ge 2)$ $(\forall k\in \{r_p+2, \ldots,r_{p+1}+1\} )$ we have 
		\begin{subequations}
			\begin{align}
				f(x_k)-f(x^*)
				& \leq \left(\frac{2^{p+1}}{(k-r_{p}) (r_p-r_{p-1}+2)...(r_2-r_1+2)(r_1+2)} \right)^2 \frac{L}{2}(x_0-x^*)^2 \label{eq:multrestrhs}
				\\
				& \le \frac{2L(x_0-x^*)^2}{(k+1)^2}.
			\end{align}
		\end{subequations}
	\end{theorem}
	\begin{proof}
		See Appendix~\ref{sec:3}.
	\end{proof}
	
	\begin{remark} 
		For sufficiently large $k$ and $r$, the upper bound of $8/((k-r)(r+2))^2$ in \eqref{eq:newbound} can be much smaller than the classical factor $2/(k+1)^2$ for $k>r$.  For example, if $r=10,000$, $k=11,000$, then the factor in our new bound is $\approx 8\cdot 10^{-14}$ versus $2/(k+1)^2\approx 1.7\cdot 10^{-8}$.  The right-hand side of \eqref{eq:multrestrhs} further improves \eqref{eq:newbound}.
	\end{remark}
	\begin{corollary}
		Let $x_0\in \R$ and set $(\forall \knn )$ $t_k=\tfrac{k+1}{2}$.
		Let $(x_k)_{k \in \mathbb{N}}$ be given by  Algorithm~1 . 
		Then 
		\begin{equation}
			f(x_k)\to f^*.
		\end{equation}
	\end{corollary}
	\begin{proof}
		This is a direct consequence of
		\cref{theorem:1Re} 
		and \cref{prop:mRe}.
	\end{proof}
	
	\section{Conclusion}
	In our main result, we prove that the gradient based restarts of O'Donoghue and Cand\`es  \cite{O'Donoghue2015Jun}
	employed with a slight modification (see Remark~\ref{rem:ourres})
	improve the classical right-hand side bound from AGD when $n=1$. 
	The modification 
	in Remark~\ref{rem:ourres}
	allowed the use of Proposition~\ref{prop:xnew_bound}, and in computational experiments it performed slightly better, even in cases where $n \geq 1.$ 
	\\
	One remark is that our proof translates to higher dimensions if one has a restart condition that implies $t_k ||x_{k+1}-x^*|| \leq ||z_{k+1}-x^*||$. While this cannot be useful in practice because it requires prior knowledge of the minimizer, we have experimentally observed that in cases where we do know the minimizer this restart condition performs well. Giselsson and Boyd also noted that having $||x_{k} - x^*|| \leq ||z_{k}-x^*||$ would improve the constant in their convergence analysis. A direction of further research would be to create a restart condition which would imply this condition. Although the analysis of  Algorithm~1  does not extend to higher dimensions, we want to remark on the case of nearly separable functions. For such functions, we observed that running  Algorithm~1  in parallel along each coordinate can perform better than using the restarting algorithm of O'Donoghue and Cand\`es as illustrated in Appendix~\ref{sec:HL}.
	\subsection*{Acknowledgements}
	The research of WMM and SAV is supported by their Natural Sciences and 
	Engineering Research Council of Canada Discovery Grants (NSERC-DG).
	\bibliography{OptML_2023/opt2023}
	\newpage
	\clearpage
	\appendix
	
	\section{}
	\label{sec:1}
		\textit{Proof of Proposition~\ref{prop:GySign}.}
		Let $j\ge 1$
		and suppose that 
		$x_{j}$ and $y_{j}$ are given by the updates
		in \cref{e:classicAGD}. 
		Using the fact that 
		$f^\prime (x^*)=0$
		yields
		\begin{equation}
			\label{eq:xkyk}
			(x_j-x^*)(y_{j-1}-x^*)=(y_{j-1}-x^*)^2
			-\tfrac{1}{L}f^\prime (y_{j-1})(y_{j-1}-x^*)\ge 0.
		\end{equation}
		Moreover, because $x_j$ is a gradient step from $y_{j-1}$ of size $1/L$ 
		we learn that 
		\begin{equation}
			\label{eq:xkykabs}
			\abs{x_{j}-x^*}\le \abs{y_{j-1}-x^*}.
		\end{equation}
		(\labelcref{prop:GySign:i}): 
		``$\Rightarrow$":
		Suppose that the gradient condition is satisfied at iteration $k$, i.e.,
		$f^\prime(y_k)(x_{k+1}-x_k) > 0$.
		Then
		for
		$j \le  k-1$ we have 
		$f^\prime(y_j)(x_{j+1}-x_j) \leq  0$.
		In particular, for $ j =k-1$ we learn that  
		\begin{equation}
			\label{eq:restconk}
			f^\prime(y_{k-1})(x_{k}-x_{k-1}) \leq  0.
		\end{equation}
		Because $x_k < x^*$,
		using \cref{eq:xkykabs} and \cref{eq:xkyk} applied with $j$ replaced by $k$
		we conclude that $y_{k-1} \le x_k < x^*.$ Therefore $f^\prime (y_{k-1}) < 0$ and consequently  by \cref{eq:restconk} we have 
		\begin{equation}
			x_k \ge x_{k-1}.
		\end{equation}
		We claim that  $y_k > x^*$. Indeed,  Algorithm~1  gives us,
		\begin{equation}
			y_k = x_k + \tfrac{t_k-1}{t_{k+1}}(x_k - x_{k-1}).
		\end{equation}
		Since $t_k$ is an increasing sequence and $t_k \geq 1$ for all $k \geq 1$, alongside the fact that $x_k - x_{k-1} \ge 0,$ we obtain that $y_k \ge  x_k.$ Suppose for eventual contradiction that 
		$y_k < x^*$. 
		Then the above arguments yields $f^\prime(y_k) < 0$ and 
		$x_{k+1} = y_{k} - \frac{1}{L}f^\prime(y_k)> y_k\ge x_k$. This implies that 
		$$f^\prime(y_k)(x_{k+1} - x_k) <0$$
		which is absurd in view of \cref{e:condprop}. Therefore we have that $y_k > x^*$ which, in turn, implies that $f^\prime(y_k) > 0$ and that 
		$x^*<x_{k+1}\le y_k$. 
		Hence we conclude $x_k < x^* < x_{k+1}.$
		``$\Leftarrow$":
		Suppose that $x_k < x^* < x_{k+1}.$ 
		Then \cref{eq:xkyk} 
		and \cref{eq:xkykabs}
		applied with $j$ replaced by $k$ and with $j$ replaced by $k+1$
		implies
		\begin{equation}
			y_{k-1} < x_k < x^* < x_{k+1} < y_k,
		\end{equation}
		which immediately implies  that 
		\begin{equation}
			f^\prime(y_k)(x_{k+1}-x_k) > 0.
		\end{equation}
		Therefore the gradient restart condition is satisfied at iteration $k$. 
		(\labelcref{prop:GySign:ii}): Proceed similarly to the proof of
		(\labelcref{prop:GySign:i}). \hfill{$\blacksquare$}
		
		\section{}
		\label{sec:2}
		
		\textit{Proof of \cref{theorem:1Re}.}
		Observe that $r\ge 2$.
		In the following we let $(\overline{x}_k)_{k \in \mathbb{N}}$, 
		$(\overline{y}_k)_{k \in \mathbb{N}}$,
		and
		$
		(\overline{t}_k)_{k \in \mathbb{N}}$ be given by the classical AGD update
		\cref{e:classicAGD} with the starting 
		point $\overline{x}_0:=x_0$.
		Clearly, $\overline{t}_0 =    {t}_0 = 1$. Moreover, for all $k \leq r$ we have, 
		\begin{equation}
			\label{e:thmmain:1}
			\overline{x}_k =    {x}_k, \overline{t}_k =    {t}_k.
		\end{equation} 
		Furthermore, since we keep the point after the restart iteration we have $\overline{x}_{r+1} =    {x}_{r+1}$. Now since at iteration $r+1$ we reset the parameter sequence we have that now $   {t}_{r+1}= 1$, which means that for $k \geq r+1$ we have the relation that $   {t}_k = \overline{t}_{k-(r+1)}.$
		For 
		$k \in  \{0,1,2\dots\}$
		we set 
		\begin{subequations}
			\begin{align}
				\overline{z}_{k+1}&:= (1-t_k)\overline{x}_{k} + \overline{t}_k \overline{x}_{k+1},
				\label{eq:defzktil}
				\\
				\overline{\delta}_k 
				&:= f(\overline{x}_k) - f(x^*)
				\label{eq:defdeltaktil}
			\end{align}
		\end{subequations}
		and 
		we set 
		\begin{subequations}
			\begin{align}
				{z}_{k+1}&:= (1-   {t}_k)   {x}_{k} +    {t}_k    {x}_{k+1},
				\label{eq:defzk}
				\\
				{\delta}_k &:= f(   {x}_k) - f(x^*).
				\label{eq:defdeltak}
			\end{align}
		\end{subequations}
		Due to the restart,
		\begin{equation}
			{z}_{r+2} = (1 -    {t}_{r+1})   {x}_{r+1} +    {t}_{r+1}    {x}_{r+2} =    {x}_{r+2}.
		\end{equation}
		Let $k \geq r+2$.  Using Fact~\ref{lemma:decProp} 
		we have 
		\begin{subequations}
			\begin{align}
				{\delta}_k  = f(   {x}_k) - f(x^*) &\leq \tfrac{1}{   {t}_{k-1}^2}\left( \tfrac{L}{2}(   {z}_{k}-x^*)^2 +    {t}^2_{k-1}   {\delta}_k\right)\\
				&\leq \tfrac{1}{   {t}_{k-1}^2}\left( \tfrac{L}{2}(   {z}_{k-1}-x^*)^2 +    {t}^2_{k-2}   {\delta}_{k-1}\right)\\
				&\leq \\
				&\vdots \\
				&\leq \tfrac{1}{   {t}_{k-1}^2}\left( \tfrac{L}{2}(   {z}_{r+2}-x^*)^2 +    {t}^2_{r+1}   {\delta}_{r+2}\right)\\
				&= \tfrac{1}{   {t}_{k-1}^2}\left( \tfrac{L}{2}(   {x}_{r+2}-x^*)^2 +    {\delta}_{r+2}\right)
				\label{eqn:TeleAGDRE}
			\end{align}
		\end{subequations}
		where we now used in the equality that $   {t}_{r+1} = 1,$ and that $   {z}_{r+2}=   {x}_{r+2}$.
		Recall from  Algorithm~1  that we update ${x}_k
		$ by taking a gradient step from ${y}_{k-1}.$ 
		$$   {x}_{r+2} =    {y}_{r+1} - \frac{1}{L}f^\prime(   {y}_{r+1})$$
		but since we restarted we have that $   {y}_{r+1} =    {x}_{r+1}$, so
		$$   {x}_{r+2} =    {x}_{r+1} - \frac{1}{L}f^\prime(   {x}_{r+1}).$$
		This observation allows us to 
		make further progress. 
		On the one hand, 
		recalling $   {y}_{r+1} =    {x}_{r+1}$,
		it follows from 
		\cref{eq:xkykabs} applied with $j$
		replaced by $r+2$
		that 
		\begin{equation}
			(   {x}_{r+2} - x^*)^2 \leq (   {x}_{r+1} - x^*)^2.
		\end{equation}
		On the other hand, 
		because $   {x}_{r+2}$
		is a gradient step from 
		$   {x}_{r+1}$ with a proper step-length
		we have that $f(\overline{x}_{r+1})=f(   {x}_{r+1}) \ge  f(   {x}_{r+2}) > f(x^*)$
		which, in view of \cref{eq:defdeltak},
		implies  that  $   {\delta}_{r+2} \le    {\delta}_{r+1}=\overline{\delta}_{r+1}$.
		Combining this and \cref{eqn:TeleAGDRE} 
		yields
		\begin{equation}
			{\delta}_{k} \leq \tfrac{1}{\overline{t}_{k-r-2}^2}\left( \tfrac{L}{2}(x_{r+1} - x^*)^2 + \overline{\delta}_{r+1}\right)
		\end{equation}
		where we used that $   {x}_{r+1} = \overline{x}_{r+1},$ and that 
		$t_{k-1} = \overline{t}_{k-{(r+1)}-1}=\overline{t}_{k-r-2}.$ Proposition~\ref{prop:xnew_bound} yields
		\begin{subequations}
			\begin{align}
				{\delta}_{k} 
				&\leq \frac{1}{\overline{t}_{k-r-2}^2}\left( \frac{L}{2}\frac{1}{\overline{t}_r^2}(\overline{z}_{r+1} - x^*)^2 + \overline{\delta}_{r+1}\right)\\
				&= \frac{1}{\overline{t}_{k-r-2}^2 \overline{t}^2_r}\left( \frac{L}{2}(\overline{z}_{r+1} - x^*)^2 + \overline{t}_r^2\overline{\delta}_{r+1}\right)\\
				&\leq\\
				&\qquad\vdots\\
				&\leq \frac{1}{\overline{t}_{k-r-2}^2 \overline{t}^2_r}\left( \tfrac{L}{2}(\overline{z}_{1} - x^*)^2 + \overline{\delta}_{1}\right)\\
				&= \frac{1}{\overline{t}_{k-r-2}^2 \overline{t}^2_r}\left( \tfrac{L}{2}(\overline{x}_{1} - x^*)^2 + \overline{\delta}_{1}\right)\\
				&\leq
				\frac{L}{2 (\overline{t}_{k-r-2}\overline{t}_r)^2}\left(x_0 - x^*\right)^2
				\label{e:40g}
				\\
				&\leq 
				\frac{8L}{((k-r) (r+2))^2}\left(x_0 - x^*\right)^2,
				\label{e:40h}
			\end{align}   
		\end{subequations}
		where \cref{e:40g} follows from 
		Fact~\ref{fact:1} 
		applied with $x$ replaced by $x^*$ and $y$
		replaced by $x_0$
		and \cref{e:40g} follows from 
		\cref{eqn:FISTAseqA}.
		To finish the proof we claim that for $k\ge r+3$ we have 
		\begin{equation}
			\label{e:ineqbd}
			\frac{4}{((k-r) (r+2))^2}\le \frac{1}{(k+1)^2}.
		\end{equation}
		Indeed, observe that 
		\begin{subequations}
			\begin{align}
				\cref{e:ineqbd}
				&\siff 
				2(k+1)\le (k-r)(r+2)
				\\
				&\siff r(k-(r+2))\ge 2,
			\end{align}   
		\end{subequations}
		which is always true for $k\ge r+3$ by recalling that $r\ge 2$.
		By definition, $ {x}_{r+2}=T ({x}_{r+1})$ and hence  $|   { x}_{r+2}-x^*|\le |   {x}_{r+1}-x^*|$ and $f(   {x}_{r+2})\le f(   {x}_{r+1})$.
		\hfill{$\blacksquare$}

		\section{}
		\label{sec:3}
		\textit{Proof of \cref{prop:mRe}.}
		We introduce the following notation to identify the multiple restarts. We denote by
		$(\xre{i}_j)_{j\in \mathbb{N}}$ the sequence of iterates given by starting AGD with $\xre{i}_0 = x_{r_i+1}$ and $t^{(r_i)}_0 = t_{r_i+1} = 1.$ We can now identify that between two consecutive restarts say $r_{i}$ and $r_{i+1}$ we have that $(\xre{i}_0,\xre{i}_1,\xre{i}_2, \dots \xre{i}_{r_{i+1}-r_i})$ coincide with $(x_{r_i+1},x_{r_i+2},x_{r_i+3},\dots,x_{r_{i+1}})$. For iterations  $k >r_p$ we have that,
		$(\xre{p}_0,\xre{p}_1,\dots,\xre{p}_{k-r_p-1})$ coincide with with $(x_{r_p+1},x_{r_p+2},\dots,x_k)$. 
		Upholding the notation of \cref{eq:def:zdelta}
		we similarly denote the corresponding values of 
		$\zre{i}$, $\delta^{(r_i)}$ and $\tre{i}$.
		We proceed similarly to the proof of \cref{theorem:1Re}. 
		Let $k \ge r_p+2$. Then
		\begin{subequations}
			\begin{align}
				f(x_k) - f(x^*) &= f(\xre{p}_{k-r_p-1})-f(x^*)\\
				&\leq  \frac{1}{(\tre{p}_{k-r_p-2})^2}\left( \frac{L}{2}(\zre{p}_{k-r_p-1} - x^*)^2 +(\tre{p}_{k-r_p-2})^2\dre{p}_{k-r_p-1} \right)\\
				&\leq \\
				\nonumber
				&\vdots
				\nonumber
				\\
				&\leq \frac{1}{(\tre{p}_{k-r_p-2})^2}\left( \frac{L}{2}(\zre{p}_{1} - x^*)^2 +(\tre{p}_{0})^2\dre{p}_{1} \right)\\
				&= \frac{1}{(\tre{p}_{k-r_p-2})^2}\left( \frac{L}{2}(\xre{p}_1 - x^*)^2 +\dre{p}_{1} \right) .  
			\end{align}
		\end{subequations}
		We now use that $\xre{p}_1$ is a gradient step from $\xre{p}_0$ to conclude,
		\begin{subequations}
			\begin{align}
				f(x_k)-f(x^*) 
				&\leq \frac{1}{(\tre{p}_{k-r_p-2})^2}\left( \frac{L}{2}(\xre{p}_0 - x^*)^2 +\dre{p}_{0} \right)\\
				&=\frac{1}{(\tre{p}_{k-r_p-2})^2}\left( \frac{L}{2}(\xre{p-1}_{r_p-r_{p-1}+1} - x^*)^2 +\dre{p-1}_{r_p-r_{p-1}+1} \right)
			\end{align}
			\label{eq:ms6}
		\end{subequations}
		where we get equality because when the $r_{p}$ restart gets triggered we keep $\xre{p-1}_{r_p - r_{p-1}+1}$ unchanged and set $\xre{p}_0 = \xre{p-1}_{r_p-r_{p-1}+1}.$ We can now use Proposition~\ref{prop:xnew_bound} to obtain
		\begin{equation}
			|\xre{p-1}_{r_p-r_{p-1}+1}-x^*| \leq \frac{1}{\tre{p-1}_{r_p-r_{p-1}}}|\zre{p-1}_{r_p-r_{p-1}+1}-x^*|.
			\label{eq:ms7}
		\end{equation}
		Combining \cref{eq:ms6} and  \cref{eq:ms7}
		and factoring out the $\frac{1}{\tre{p-1}_{r_p-r_{p-1}}}$ term we obtain
		\begin{equation}
			f(x_k)-f(x^*) 
			\leq \frac{1}{(\tre{p}_{k-r_p-2})^2(\tre{p-1}_{r_p-r_{p-1}})^2}\left( \frac{L}{2}(\zre{p-1}_{r_p-r_{p-1}+1} - x^*)^2 +(\tre{p-1}_{r_p-r_{p-1}})^2\dre{p-1}_{r_p-r_{p-1}+1} \right).
			\label{eq:msi}
		\end{equation}
		We now repeat the same procedure. Since the terms in the parentheses on the right-hand side of \cref{eq:msi} have the monotonicity property from Fact~\ref{lemma:decProp}, we use it repeatedly
		until we obtain an expression 
		of the form of the right-hand side of in \cref{eq:mono}
		that features 
		$\zre{i}_1$.
		We then apply the fact that $\zre{i}_1 = \xre{i}_1$, which is a gradient step from $\xre{i}_0$, and finally use Proposition~\ref{prop:xnew_bound}. Observe that every time
		we apply Proposition~\ref{prop:xnew_bound} an additional factor of $1/(\tre{i}_{r_{i+1}-r_1})^2$ for some $i\in \{1,\ldots, p\}$ appears in the denominator of the right-hand side of 
		the inequality  \cref{eq:msi}. This is done $p$ times until the term in the parentheses reduces to $(L/2)(x_0-x^*)^2.$
		We therefore obtain
		\begin{equation}\label{eqn:tkBNDS}
			f(x_k)-f(x^*) \leq
			\frac{1}{(\tre{p}_{k-r_p-2})^2\left( \prod_{i=1}^{p}(\tre{i}_{r_{i}-r_{i-1}})^2\right)}\left(\frac{L}{2}(x_0-x^*)^2\right).
		\end{equation}
		Recall that for each $i\in \{1, \ldots, p\}$
		$(\forall k\in \{0, \ldots,r_i \})$
		we have $\tre{i}_k= (k+2)/2,$ hence:
		
		\begin{equation}\label{eqn:FistaPropRP}
			\frac{1}{\tre{p}_{k-r_p-2}} =\frac{2}{k-r_p }
		\end{equation}
		and
		\begin{equation}\label{eqn:FistaPropRi}
			\frac{1}{\tre{i}_{r_i - r_{i-1}}} = \frac{2}{r_{i}-r_{i-1}+2}.
		\end{equation}
		Now let $\knn$. Observe that
		if $k\le r_2+1$ we obtain 
		the $\mathcal{O}(1/k^2)$
		function value convergence rate
		from \cref{theorem:1Re}.
		We now show that 
		$(\forall p\ge 2)$
		$(\forall k \geq r_p +2)$ we have 
		\begin{equation}
			\frac{2^{p+1}}{(k-r_p)(r_p-r_{p-1}+2)\dots (r_2-r_1+2)(r_1+2)} \leq \dots \leq \frac{4}{(k-r_1)(r_1+2)}\leq \frac{2}{(k+1)}.
		\end{equation} 
		We proceed by induction on $p$.
		We first verify the base case at $p =2$. Applying \cref{eqn:FistaPropRP} and \cref{eqn:FistaPropRi} to \cref{eqn:tkBNDS} we get
		\begin{equation}
			f(x_k) - f(x^*) \leq \left(\frac{8}{(k-r_2)(r_2-r_1+2)(r_1 +2)}\right)^2\frac{L}{2}(x_0 -x^*)^2.
		\end{equation}
		For $k = r_2+2$ we have 
		\begin{subequations}
			\begin{align}
				\frac{8}{(k-r_2)(r_2-r_1+2)(r_1+2)} &= \frac{8}{(r_2+2-r_2)(r_2-r_1+2)(r_1+2)}\\
				&= \frac{8}{(2)((r_2+2)-r_1)(r_1+2))}\\
				&=\frac{4}{(k-r_1)(r_1+2)}\le \frac{2}{k+1},
			\end{align}
			\label{eq:msii}
		\end{subequations}
		where the inequality follows from 
		the fact that $r_1\ge 2$ 
		and $k=r_2+2\ge r_1+3$.
		Now let $k \geq r_2+3$.
		Then
		\begin{subequations}
			\begin{align}
				&\qquad\frac{8}{(k-r_2)(r_2-r_1+2)(r_1+2)} - \frac{4}{(k-r_1)(r_1+2)}
				\nonumber
				\\
				&= \frac{4}{(r_1+2)}\left( \frac{2}{(k-r_2)(r_2-r_1+2)}-\frac{1}{k-r_1}\right)\\
				&=\frac{4}{(r_1+2)}\left( \frac{2(k-r_1)}{(k-r_2)(r_2-r_1+2)(k-r_1)}-\frac{(k-r_2)(r_2-r_1+2)}{(k-r_2)(r_2-r_1+2)(k-r_1)}\right)\\
				&=\frac{4}{(r_1+2)}\left( \frac{2(k-r_1) - (k-r_2)(r_2-r_1+2)}{(k-r_2)(r_2-r_1+2)(k-r_1)}\right).
				\label{eq:ms4}
			\end{align}
			\label{eq:ms4i}
		\end{subequations}
		Write $k = r_2+a$ where $a \geq 3$. 
		It is sufficient to show that the numerator of the right-hand side 
		of \cref{eq:ms4} is nonpositive. 
		To this end we have 
		\begin{subequations}
			\begin{align}
				2(k-r_1) - (k-r_2)(r_2-r_1+2) &= 2(r_2+a-r_1) - (r_2+a-r_2)(r_2-r_1+2)\\
				&=2a (r_2-r_1) - 2a -a(r_2-r_1)\\
				&= (2-a)(r_2-r_1)
				\leq 0.
			\end{align}
			\label{eq:ms5}
		\end{subequations}
		Combining \cref{eq:msii},
		\cref{eq:ms4i}, and 
		\cref{eq:ms5}
		we conclude that for $k \geq r_2 +2$ we have
		$$\frac{8}{(k-r_2)(r_2-r_1+2)(r_1+2)} \leq \frac{4}{(k-r_1)(r_1+2)} \leq \frac{2}{k+1}.$$
		This proves the base case.
		Now suppose that we have restarted $p$ times, $p\ge 2$, and for $k \geq r_p +2$ the following inequalities hold
		\begin{equation}
			\frac{2^{p+1}}{(k-r_p)(r_p-r_{p-1}+2)\dots (r_2-r_1+2)(r_1+2)} \leq \dots \leq \frac{4}{(k-r_1)(r_1+2)}\leq \frac{2}{(k+1)}.
		\end{equation} 
		Consider the iterations $k \geq r_{p+1} +2$ where $r_{p+1}$ is the iteration of the $p+1$ restart.
		If we apply \cref{eqn:FistaPropRP} and \cref{eqn:FistaPropRi} to \cref{eqn:tkBNDS} we obtain, 
		$$f(x_k)-f(x^*) \leq \left(\frac{2^{p+2}}{(k-r_{p+1})\prod_{i=1}^{p+1} (r_i-r_{i-1}+2))} \right)^2 \frac{L}{2}(x_0-x^*)^2.$$
		We are concerned with iterations $k \geq r_{p+1}+2$ since those would be generated after the $p+1$ restart. At $k = r_{p+1}+2$ we have
		\begin{subequations}
			\begin{align}
				\frac{2^{p+2}}{(k-r_{p+1})\prod_{i=1}^{p+1} (r_i-r_{i-1}+2)} &= \frac{2^{p+2}}{(r_{p+1}+2-r_{p+1})\prod_{i=1}^{p+1} (r_i-r_{i-1}+2)}  \\
				&=\frac{2^{p+2}}{2(r_{p+1}-r_{p}+2)\prod_{i=1}^{p} (r_i-r_{i-1}+2)}  \\
				&=    \frac{2^{p+1}}{(k-r_{p})\prod_{i=1}^{p} (r_i-r_{i-1}+2)} . 
			\end{align}
		\end{subequations}
		Observe that this is exactly what we obtain when we have $p$ restarts.
		Therefore, by the inductive hypothesis, we know that for $k = r_{p+1}+2$ we satisfy the upper bound at $x_{r_{p+1}+2}$. For $k \geq r_{p+1} +3$ we examine 
		\begin{equation}
			\frac{2^{p+2}}{(k-r_{p+1})\prod_{i=1}^{p+1} (r_i-r_{i-1}+2))} -  \frac{2^{p+1}}{(k-r_{p})\prod_{i=1}^{p} (r_i-r_{i-1}+2))} .
		\end{equation}
		Proceeding similarly to the arguments in the base case
		(see \cref{eq:ms4i} and \cref{eq:ms5}), we conclude that it is sufficient to examine the sign of
		$2(k-r_p) - (k-r_{p+1})(r_{p+1}-r_p +2)$.
		Write $k = r_{p+1} + a$ where $a \geq 3$,  we now examine
		\begin{subequations}
			\begin{align}
				&\qquad 2(k-r_p) - (k-r_{p+1})(r_{p+1}-r_p +2)
				\nonumber
				\\
				&=2(r_{p+1} + a-r_p) - (r_{p+1} + a-r_{p+1})(r_{p+1}-r_p +2) \\
				&= 2(r_{p+1} + a-r_p) - (a)(r_{p+1}-r_p +2)\\
				&=2a + 2(r_{p+1}-r_p) - 2a -a(r_{p+1}-r_p)\\
				&= (2-a)(r_{p+1}-r_p)\leq 0.
			\end{align}
		\end{subequations}
		Hence we conclude for $k\geq r_{p+1}+2$
		$$\frac{2^{p+2}}{(k-r_{p+1})\prod_{i=1}^{p+1} (r_i-r_{i-1}+2))} \leq  \frac{2^{p+1}}{(k-r_{p})\prod_{i=1}^{p} (r_i-r_{i-1}+2))}\le \frac{2}{k+1}, $$
		where the second inequality follows from the inductive hypothesis.
		Altogether we have shown that $(\forall p\ge 2)$, $(\forall k\in \{r_p+2, \ldots,r_{p+1}+1\} )$ we have \begin{subequations}
			\begin{align}
				f(x_k)  - f(x^*) &\leq \left(\frac{2^{p+1}}{(k-r_p)(r_p-r_{p-1}+2)\dots(r_2-r_1+2)(r_1+2)}\right)^2\frac{L}{2}(x_0-x^*)^2 
				\\
				&\leq \frac{2L}{(k+1)^2}(x_0-x^*)^2.
			\end{align}  
		\end{subequations}
		The proof is complete.
		\hfill{$\blacksquare$}
		\section{Numerical example: Simple Quadratic}
		\label{sec:D}
		In this and the upcoming appendices, we provide some numerical examples using the gradient-based restart strategy. The examples indicate that choosing $x_{k+1}$ as the new initial point after the restart does not hamper the algorithm.\\
		\\
		The first example is a simple quadratic  $f(x) = \frac{1}{2}x^TQx-q^Tx$, where $Q$ is an $n \times n$ positive definite matrix obtained by $Q = Q_0 + Q_0^T + 50 \Id$ where $Q_0$ is generated using a $U[0,1]$ distribution, and $q \in \R^n$ is randomly generated from a standard normal distribution. The Lipschitz constant $L$ is given by the maximum eigenvalue of $Q.$ For this experiment we set $n = 500$.
		\begin{figure}[h]
			\centering
			\includegraphics[width = \textwidth]{ 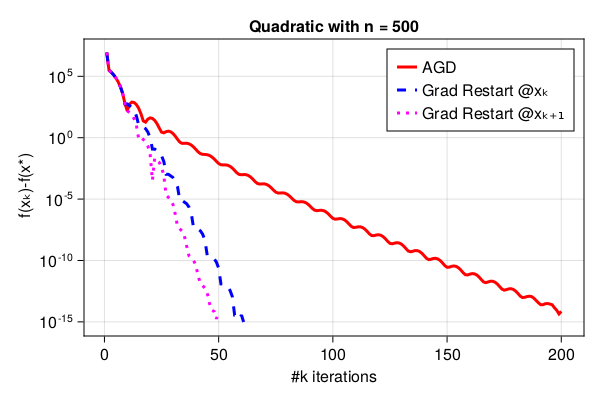}
			\caption{A \texttt{Julia} plot.
				A comparison of
				the performances of 
				the restart scheme in Algorithm~1 at $x_{k+1}$ (the pink dotted curve),
				the restart scheme in \cite{O'Donoghue2015Jun}
				at $x_k$ (the blue dashed curve)
				and the classical AGD 
				(the solid red curve)
				when applied to minimize simple quadratic.}
			\label{fig:AGD_Quad}
		\end{figure}
		We can see in \cref{fig:AGD_Quad} that keeping $x_{k+1}$ at the restart yields a modest improvement over keeping $x_k$.  Our theory for the $n=1$ case required keeping $x_{k+1}$ rather than $x_k$.
		
		\section{Huber Regression}
		\label{sec:E}
		Another numerical example is the following problem of Huber Regression. Define,
		\begin{equation} \psi_\tau (x) = 
			\begin{cases}
				x^2, & |x| \leq \tau\\
				2\tau x - \tau^2 & x\geq \tau\\
				-2\tau x - \tau^2 & x\leq - \tau.
			\end{cases}
		\end{equation}
		Given $A \in \R^{m \times n}$ and $y \in \R^m$ we consider the optimization problem,
		\begin{equation}
			\min_{x \in \R^n} \frac{1}{2} \sum_{i=1}^m \psi_\tau (a_ix - y_i)
		\end{equation}
		where $a_i$ is a row of $A.$ For this experiment we set $\tau = 0.5$, $m=300, n=50,$ and $A$ and $y$ were randomly generated using a standard normal distribution.
		\begin{figure}[h]
			\centering
			\includegraphics[width = \textwidth]{ 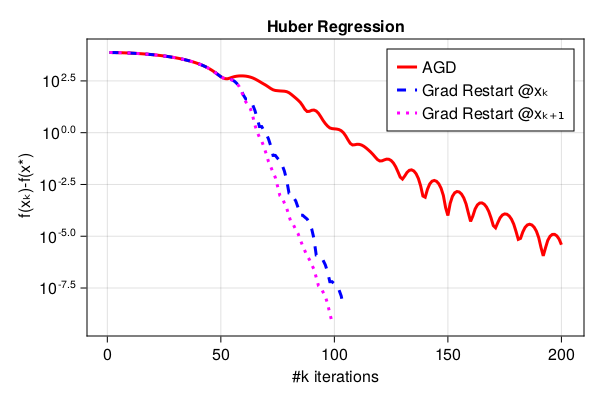}
			\caption{A \texttt{Julia} plot.
				A comparison of
				the performances of 
				the restart scheme in Algorithm~1 at $x_{k+1}$ (the pink dotted curve),
				the restart scheme in \cite{O'Donoghue2015Jun}
				at $x_k$ (the blue dashed curve)
				and the classical AGD 
				(the solid red curve)
				when applied to solve the Huber Regression problem.}
			\label{fig:AGD_Huber}
		\end{figure}
		In \cref{fig:AGD_Huber} it can be observed that restarting using $x_{k+1}$ offers slightly better performance.
		
		\section{The Hinder-Lubin Example Expanded}
		\label{sec:HL}
		
		Hinder and Lubin constructed a function \cite[Appendix D.4]{hinder2020generic} for which the restarts of O'Donoghue and Cand\'es performed poorly until the iterates of $(x_k)_\knn$ were close to the minimizer. Their objective function is defined as
		\begin{equation}
			f(x) = \sum_{i=1}^n i h_\delta (x_i) + \frac{\alpha}{2}||x||^2_2,
			\label{eq:hinderlubin}
		\end{equation}
		where
		\begin{equation}
			h_\delta(z) = \begin{cases}
				z^2/2, & z \geq -\delta\\
				-\delta z -\delta^2/2 & z<-\delta.
			\end{cases}
		\end{equation}
		Remark that $f: \R^n \rightarrow \R$ is $(n+\alpha)-$smooth and $\alpha$-strongly convex with a unique minimizer at $x = 0.$ The function $f$ is also separable, and therefore one can apply  Algorithm~1  along each coordinate. 
		We see that the restart
		scheme of  Algorithm~1 performs well. In most cases, we do not have knowledge of whether the function is separable or not but there are cases where running  Algorithm~1  in parallel along each coordinate has merit. We introduce a modified function obtained from \eqref{eq:hinderlubin} that is not separable. Given a matrix $A\in \R^{m \times n}$, define
		\begin{equation}
			F(x) = f(x) + \gamma \sum_{i=1}^m \left(\scal{a_i}{x} + \sqrt{\scal{a_i}{x}^2+1}\right)
			\label{eq:mhinderlubin}
		\end{equation}
		where $f$ is the Hinder and Lubin function defined in 
		\cref{eq:hinderlubin}, $a_i$ are the rows of $A$, and $\gamma \in \R$. For the experiment we chose $m = 110,n=100,$ and set $\delta = \alpha = \gamma = 10^{-4}.$ Note that $F$ is now $\left(n + \alpha + \gamma \sum_{i=1}^m \norm{a_i}\right)-$smooth. In the experiment we compare AGD, AGD with the gradient restarts, and  Algorithm~1  running along each coordinate. 
		\begin{figure}[ht]
			\centering
			\includegraphics[width = \textwidth]{ 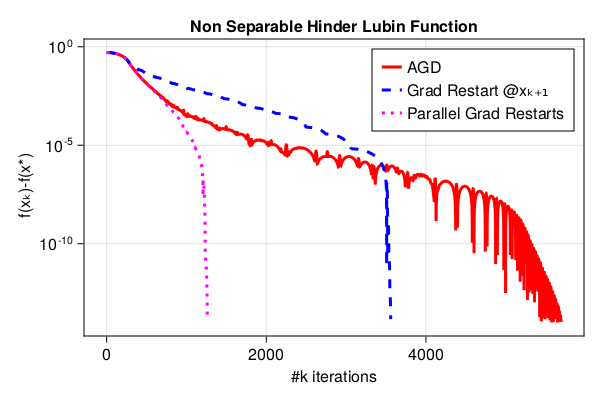}
			\caption{A \texttt{Julia} plot.
				A comparison of
				the performances of 
				the parallel gradient restart scheme using  Algorithm~1 for each coordinate  (the pink dotted curve),
				the restart scheme in \cite{O'Donoghue2015Jun}
				at $x_{k+1}$ (the blue dashed curve)
				and the classical AGD 
				(the red curve)
				when applied to solve the  modified Hinder-Lubin function
				\cref{eq:mhinderlubin}. As the plots reflect, while the objective function is not separable, separate restarts still improve performance.}
			\label{fig:HiLuMod}
		\end{figure}
		In \cref{fig:HiLuMod} it can be seen that the gradient based restarts perform poorly until the iterates get close to the minimizer. On the other hand, restarting along each coordinate performs much better even though $F$ is not separable. 
	\end{document}